\documentclass[reqno]{amsart}
\AtBeginDocument{}
\usepackage[colorlinks=true]{hyperref}
\hypersetup{urlcolor=blue, citecolor=red}
\allowdisplaybreaks

\RequirePackage[top=3truecm,bottom=2.5truecm,left=3.3truecm,right=3.3truecm,includefoot, xetex]{geometry}
\usepackage{amssymb}

\begin{document}
	\title[]
	{Global dynamics of a predator-prey model with alarm-taxis}
	
	\author[]
	{Songzhi Li and Kaiqiang Wang*}
	\thanks{* Corresponding author.}

	\address{Songzhi Li  \hfill\break
		School of Mathematics, Jilin University, Changchun 130012, China}
	\email{lisz20@mails.jlu.edu.cn}
	
	\address{Kaiqiang Wang \hfill\break
		School of Mathematics, Jilin University, Changchun 130012, China}
	\email{wangkq20@mails.jlu.edu.cn}
	\date{}
	\thanks{}

	\begin{abstract}
		This paper concerns with the global dynamics of classical solutions to 
		an important alarm-taxis ecosystem, which demonstrates the 
		behaviors of prey that attract secondary predator when threatened by 
		primary predator. And the secondary predator pursues the signal 
		generated by the interaction of the prey and primary predator.  
		However, it seems that the necessary gradient estimates for global 
		existence cannot be obtained in critical case due to 
		strong coupled structure. Thereby, we develop a new approach to estimate 
		the gradient of prey and primary predator which takes advantage of 
		slightly higher damping power. Then the boundedness of classical 
		solutions in two dimension with Neumann boundary conditions can be 
		established by energy estimates and semigroup theory. Moreover, by 
		constructing Lyapunov functional, it is proved that the coexistence 
		homogeneous steady states is asymptotically stability and the 
		convergence rate is exponential under certain 
		assumptions on the system coefficients.
		
	\end{abstract}
	
	\maketitle
	\let\oldsection\section
	\renewcommand\section{\setcounter{equation}{0}\oldsection}
	\renewcommand\thesection{\arabic{section}}
	\renewcommand\theequation{\thesection.\arabic{equation}}
	\newtheorem{theorem}{\indent Theorem}[section]
	\newtheorem{lemma}[theorem]{\indent Lemma}
	\newtheorem{proposition}[theorem]{\indent Proposition}
	\newtheorem{definition}{\indent Definition}[section]
	\newtheorem{remark}{\indent Remark}[section]
	\newtheorem{corollary}{\indent Corollary}[section]
	\newtheorem{claim}{\indent Claim}
	
	\allowdisplaybreaks
	
	\section{Introduction and main results}
	In the nature, many species emit alarm calls when approached by predators, 
	which are an important mechanism of anti-predator behavior in the 
	ecological food web \cite{KS1}. Alarm call signals may be chemical, 
	acoustic, visible activities of movement or any other forms that might be 
	detectable by the receiver. There is a hypothesis named burglar alarm which 
	is an indirect anti-predator behavior and a form of the attracting a 
	secondary predator hypothesis whereby the secondary predator feeds on the 
	primary predator \cite{CBS}. As a consequence, it could improve prey's 
	chances of survival despite secondary predator's threat. This attraction of 
	a secondary predator has been observed in, for instance, the marine 
	environment in chemical release by dinoflagellate bioluminescence \cite{AT} or the 
	interactions among plants, herbivorous mites, and predatory mites \cite{DSBP}. To 
	verify this hypothesis and pursue mathematical analysis, the 
	one-dimensional alarm-taxis model was recently proposed in \cite{HB}:
	\begin{equation}\label{origin}
		\left\{
		\begin{aligned}
			&u_t = \eta d u_{xx} + f(u,v,w), & x& \in (0,L),\ t>0,\\
			& v_t = (dv_x-\xi vu_x)_x + g(u,v,w), & x& \in (0,L),\   t>0,\\
			&w_t=(w_x-\chi w \phi_x (u,v))_x + h(u,v,w), & x& \in (0,L),\   t>0,\\
			&u_x=v_x-\xi vu_x=w_x-\chi \phi_x(u,v)=0, & x&\in 0, L, \  t>0,\\
			& (u,v,w)(x,0)=(u_0,v_0,w_0)(x)>0, & x & \in (0,L),
		\end{aligned}\right.
	\end{equation}
where $ u,v,w $ represent resource or prey (e.g. dinoflagellates), primary 
predator (e.g. copepods), and secondary predator (e.g. fish), respectively; $ 
\eta, d, \xi, \chi, s, q, b_1, b_2, c_1, c_2 >0 $ and $ b_3, c_3 \ge 0 $ are 
constants. The interaction of the prey resource and primary predator creates a 
signal whose intensity is given by $ \phi (u,v)>0 $. Whilst there are many ways 
one could represent this signal intensity and one of the plausible assumptions 
is that it is proportional to the encounter rate between the prey and the 
predator, i.e. $ \phi (u,v) \propto uv $. When $ f(u,v,w)=u(1-u-b_1v-b_3w) $, $ 
g(u,v,w)=v(s(1-v)+c_1u-b_2w) $ and $ h(u,v,w)=w(q(1-w)+c_2v+c_3u) $, the 
existence of global bounded solutions of \eqref{origin} was established in 
\cite{HB}, and the global stability of coexistence steady state for the special 
case $ b_3=c_3=0 $ was shown under certain conditions. To further study the 
multi-dimensional form, \eqref{origin} can be rewritten as
	\begin{equation}\label{multi}
			\left\{
		\begin{aligned}
			&u_t = d_1\Delta u + \mu_1 u(1-u) -b_1uv-b_3 f_1(u,w),&x\in\ & \Omega, t>0,\\
			&v_t = d_2\Delta v -\nabla \cdot (\xi v \nabla u)+ \mu_2 v(1-v) + uv -b_2vw,&x\in \ &\Omega, t>0,\\
			&w_t= \Delta w - \nabla \cdot [\chi w\nabla(uv)]+\mu_3 w(1-w)+vw+c_3f_1(u,w),&x\in \ &\Omega, t>0,\\
			&\frac{\partial u}{\partial \nu} = \frac{\partial v}{\partial \nu} =\frac{\partial w}{\partial \nu} =0, &x\in \ &\partial \Omega, t>0,\\
			&u(x,0)=u_0(x), v(x,0)=v_0(x), w(x,0)=w_0(x), &x\in \ &\Omega
		\end{aligned}
		\right.
	\end{equation}
	in a bounded smooth domain $ \Omega \subset \mathbb{R}^N $ with parameters $ d_1, d_2, \mu_1, \mu_2, \mu_3, b_1, b_2, \xi, \chi >0 $ and $ b_3, c_3 \ge 0 $. Moreover, $ f_1(u,w) $ represents different functional response function. The relative results have been obtained in \cite{JWW} when $ N=2 $ and $ f_1(u,w)= \frac{uw}{u+w} $. They established the global existence and boundedness of classical solution to \eqref{multi} and studied the global stability for the case of food chain and the case of intraguild predation.
	
	Inspired by aforementioned discussion, we focus on following strong coupled alarm-taxis system, which greatly indicates the interaction among prey, primary predator and secondary predator.
	\begin{equation}\label{model}
		\left\{
		\begin{aligned}
			&u_t = d_1\Delta u + r_1u(1-u) -b_1uv-b_3uw,&x\in\ & \Omega, t>0,\\
			&v_t = d_2\Delta v -\nabla \cdot (\xi v \nabla u)+ r_2 v(1-v) + uv -b_2vw,&x\in \ &\Omega, t>0,\\
			&w_t= \Delta w - \nabla \cdot [\chi w\nabla (uv)]+r_3 w(1-w^\sigma)+vw+uw,&x\in \ &\Omega, t>0,\\
			&\frac{\partial u}{\partial \nu} = \frac{\partial v}{\partial \nu} =\frac{\partial w}{\partial \nu} =0, &x\in \ &\partial \Omega, t>0,\\
			&u(x,0)=u_0(x), v(x,0)=v_0(x), w(x,0)=w_0(x), &x\in \ &\Omega,
		\end{aligned}
		\right.
	\end{equation}
	where $ \Omega $ is a bounded smooth domain in $ \mathbb{R}^2 $ denoting the habitat that the species reside, $ \partial \Omega $ is the boundary of $ \Omega $ and $ \nu $ denotes the outward unit normal vector on $ \partial \Omega $. The functions and parameters have the following biological meanings: $ u $, $ v $ and $ w $ represent density of prey, primary predator and secondary predator, respectively; $ d_1 $ and $ d_2 $ are positive constants representing the diffusion rates of prey and primary predator, respectively; the positive constants $ \xi $ and $ \chi $ are defined as the prey-taxis and alarm-taxis coefficients, respectively; $ r_1, r_2, r_3>0 $ are referred to intrinsic growth rates of prey, primary predator and secondary predator, respectively. It is worth noting that we replace the usual logistic term $ r_3 w (1-w) $ with $ r_3 w (1-w^\sigma) $ and intend to explore how strong the intra-species competition exponent will ensure the global boundedness of solutions.
	
	This alarm-taxis  model \eqref{model} can be regarded as an extension of the 
	two-species predator-pray system with prey-taxis originally proposed in 
	\cite{KO}. The prototypical prey-taxis model can be written as
	\begin{equation}\label{prey-taxis}
		\left\{
		\begin{aligned}
			u_t&=\Delta u-\nabla\cdot(\rho(u,v)\nabla v)+\gamma uF(v)-uh(u),  
			&x\in\Omega, t>0,\\
			v_t&=D\Delta v-uF(v)+f(v), &x\in\Omega, t>0.
		\end{aligned}
		\right.
	\end{equation}
	In  recent  years, the  global  dynamics  of  \eqref{prey-taxis}  have  been  widely 
	studied from different analytical perspectives for concrete functional 
	response functions (cf.  \cite{ABN,HZ, JW2,DL,T,W2,WSW,WWS} and  references  
	therein). On the other hand, many scholars paid much attention to 
	three-species 
	predator-prey systems. For example, the food chain model have been studied 
	by delicate coupled energy estimates in \cite{JW}; the global boundedness 
	of two-prey with competition (repulsion-taxis) and one-predator model was 
	also established in \cite{LW}. For other results of predator-prey model, we 
	refer to  \cite{RL,TW6,WW,WLS} and  references  
	therein.
	
	Our first result is the global existence of solutions with uniform-in-time bound.
	\begin{theorem}\label{GS}
		Let $ \Omega \subset \mathbb{R}^2 $ be a bounded domain with smooth boundary and all the parameters be positive constants. Assume that $ u_0\in W^{2,\infty}(\Omega) $, $ (v_0,w_0)\in [W^{1,\infty}(\Omega)]^2 $ with $ u_0,v_0,w_0 \gneqq 0 $ and $ \sigma>1 $. Then the problem \eqref{model} has a unique global classical solution $ (u,v,w)\in [C^0(\bar \Omega \times [0,\infty)\cap C^{2,1}(\bar \Omega \times (0,\infty)) ]^3 $ satisfying $ u,v,w>0 $ for all $ t>0 $. Moreover, the solution satisfies
		\begin{equation*}
			\| u(\cdot, t) \|_{W^{1,\infty}} +\| v(\cdot, t) \|_{W^{1,\infty}} + \| w(\cdot, t) \|_{L^\infty}  \leq C \quad \mathrm{for}\ \mathrm{all}\  t > 0,
		\end{equation*} 
		where $ C>0 $ is a constant independent of $ t $.
	\end{theorem}

In this paper, due to the benefits of spatio-temporal estimates possessed by 
source terms, we shall have the opportunity to take a step-by-step approach to 
complete the proof of the theorem. In the first half of the proof, our goal is 
to derive the global estimate of $ \| v(\cdot,t)\|_{L^\infty} $, for which we 
primarily use the multiplier method to get the boundeness of $ \| \nabla 
u(\cdot,t)\|_{L^2} $. Then, combining the sobolev embedding theorem in two 
dimensions, we use the spatio-temporal estimate brought by $ \| v(\cdot,t)\ln v 
(\cdot,t)\|_{L^1} $ to reach the $ L^p $ bound of $ v $. With these results, 
the important gradient estimate of $ u $ can be obtained (see Lemma 
\ref{le-K8}), which paves the way for subsequent estimates and yields the bound 
of $ \| v (\cdot,t)\|_{L^\infty} $. In the second half of the proof, we 
directly get the global estimate of $ \| \nabla v(\cdot,t)\|_{L^2} $ (see Lemma 
\ref{le-K10-K11}) in view of Lemma \ref{le-K8}  which enlightens the further 
proof and fulfill entropy estimate as well as the estimate of $ 
\int_{t}^{t+\tau} \int_{\Omega} \frac{|\nabla w(\cdot,s)|^2}{w(\cdot,s)} $. 
This helps us avoid complex joint estimates, which may make the condition of $ 
\sigma $ poor. Utilizing the same method as in the first half, the $ L^p $ bound 
of $ w $ can be readily derived that immediately complete the proof by 
collecting regularity theory and semigroup method.

We seize this chance to talk about some applications and expectations for our result. To this end, we consider following general system and corresponding models.

\begin{equation}
	\left\{
	\begin{aligned}
		&u_t = d_1\Delta u + r_1u(1-u) -uF_1(v,w),&x\in\ & \Omega, t>0,\\
		&v_t = d_2\Delta v -\nabla \cdot (\xi v \nabla u)+ r_2 v(1-v) -vF_2(u,w),&x\in \ &\Omega, t>0,\\
		&w_t= \Delta w - \nabla \cdot [w \nabla G(u,v)]+r_3 w(1-w^\sigma)+w F_3(u,v),&x\in \ &\Omega, t>0,\\
		&\frac{\partial u}{\partial \nu} = \frac{\partial v}{\partial \nu} =\frac{\partial w}{\partial \nu} =0, &x\in \ &\partial \Omega, t>0,\\
		&u(x,0)=u_0(x), v(x,0)=v_0(x), w(x,0)=w_0(x), &x\in \ &\Omega,
	\end{aligned}
	\right.
\end{equation}

$ \bullet $ Two-prey with competition (repulsion-taxis) and one predator model: when $ F_1(v,w) =  hv+a_1w $, $ F_2(u,w)=kv+a_2w $, $ F_3(u,v)=u+v $, $ \xi <0 $ and $ G(u,v)= \chi_1u + \chi_2 v $.

$ \bullet $ Two-predator with competition (repulsion-taxis) and one prey model: when $ F_1(v,w) = a_1v+a_2w $, $ F_2(u,w)=-u+hw $, $ F_3(u,v)= u-kv $, $ \xi >0 $ and $ G(u,v)= \chi_1 u - \chi_2v $.

$ \bullet $ Coupled food chain model: when $ F_1 (v,w)= b_1v+b_3w $, $ F_2(u,w) = -u+b_2w $, $ F_3(u,v)= u+v $, $ \xi >0$ and $ G(u,v)=\chi_1 u+\chi_2 v $.

$ \bullet $ General alarm-taxis model: when $ F_1 (v,w)= b_1v+b_3w $, $ F_2(u,w) = -u+b_2w $, $ F_3(u,v)= u+v $, $ \xi >0$ and $ G(u,v)=\chi u^\alpha v^\beta (\alpha, \beta \ge 1) $.

\begin{remark}
	 Using the similar strategy to our proof,  the above four kinds of models have the same result as Theorem \ref{GS} when $ \sigma>1 $.
\end{remark}

\begin{remark}
	When $ \sigma = 1 $, the crucial gradient estimate of $ u $ does not work, so how to get the global boundedness of strong coupled predator-prey system with standard logistic term is still an open problem.
\end{remark}

Our second main result concerns with the asymptotic behavior of solution to 
\eqref{model}. For convenience, we assume that $ r_1=r_2=r_3=1 $ without loss of generality. Since there are logistic terms in all three equations, we 
consider the constant steady state of three-species co-existence, which 
satisfies the following problem
\begin{equation}\label{SS}
	\left\{
	\begin{aligned}
		&u+b_1 v + b_3 w =1,\\
		&-u+v+b_2 w =1,\\
		&w^\sigma -u-v=1.
	\end{aligned}
	\right.\  \Longleftrightarrow \ 
	\left\{
	\begin{aligned}
		& u = \frac{(1-b_1)-(b_3-b_1b_2)w}{b_1+1},\\
		& v= \frac{2-(b_2+b_3)w}{b_1+1},\\
		& (b_1+1)w^\sigma + (b_2+2b_3-b_1b_2)w-4=0.
	\end{aligned}
	\right.
\end{equation}

We shall assume the parameters $ b_1, b_2, b_3 > 0 $ satisfy the following hypothesis:
\begin{equation}\tag{H}\label{H}
	b_1\le 1,\  b_3<b_1b_2,\  b_2+b_3 \le \frac{1}{2}.
\end{equation}

Next, we will explain the rationality of the existence of positive steady state $ (u_*, v_*, w_*) $. First of all, letting $ J(w):= (b_1+1)w^\sigma + (b_2+2b_3-b_1b_2)w-4 $, then $ J(0)=-4 $ and $ J(M) $ is sufficiently large when $ M>0 $ is sufficiently large. So there exists only one $ w_*>0 $ satisfying the third equation of \eqref{SS} according the zero point theorem and monotonicity. Secondly, \eqref{H} implies $ (b_2+2b_3-b_1b_2) \ge 0 $, which shows $ w<4 $. Furthermore, we have $ (1-b_1)-(b_3-b_1b_2)w > 0 $ and $ 2-(b_2+b_3)w> 2-\frac{1}{2}\times 4 =0 $ that verify the existence of $ (u_*, v_*, w_*) $.

The second result of present paper can be stated in the forthcoming theorem.
\begin{theorem}\label{GST}
	Let the assumptions in Theorem \ref{GS} and \eqref{H} hold, as well as $ (u,v,w) $ be the classical solution of \eqref{model}. Assume that $ (u_*, v_*, w_*) $ is the positive three-species coexistence steady state satisfying \eqref{SS}. If
	\begin{equation}\label{bbb}
		(b_1b_2-b_3)^2<4b_2b_3,
	\end{equation}
	then there exist $ \chi_1>0 $ and $ \xi_1>0 $ such that whenever $ 
	\chi\in(0,\chi_1) $ and $ \xi\in(0,\xi_1) $ 
	\begin{equation}
		\|u-u_*\|_{L^\infty}+\|v-v_*\|_{L^\infty}+\|w-w_*\|_{L^\infty}\le 
		C_1e^{-C_2t}, \quad \mathrm{for}\ \mathrm{all}\  t>t_0,
	\end{equation}
	with some $ t_0>1 $, where $ C_1,C_2>0 $ are two constants independent of $ t $.
\end{theorem}

\begin{remark}
	Notice that it is feasible when $ b_1, b_2, b_3 $ are subjected to \eqref{H} and \eqref{bbb}, for example, $ b_1=0.5, b_2=0.4, b_3=0.1  $. 
\end{remark}
	
	The outline of this paper is as follows. In Section \ref{2}, we will give 
	some results including the local existence theorem and some elemental estimates 
	of system \eqref{model} which are used frequently for our further proof. In 
	Section \ref{3}, we show some essential estimates and then obtain the 
	global boundedness of alarm-taxis model. Finally, we prove the global 
	stability of coexistence steady states under appropriate assumptions by 
	constructing Lyapunov functional in Section \ref{4}.

	\section{Local existence and preliminaries}	\label{2}
	In what follows, we shall abbreviate $ \int_{\Omega} f dx $ as $ \int_\Omega f $ and $ \int_{t_1}^{t_2} \int_{\Omega} f dx ds $ as $ \int_{t_1}^{t_2}\int_{\Omega} f $ for the sake of  convenience without loss of generality. Moreover, we use  $ K_i $ and  $ c_i $ $ (i=1,2, ...) $ to denote a different upper bound in subsequent lemmas and generic constant independent of $ t $ which  may vary in each proof, respectively. We first give the local existence of \eqref{model}, which can be readily proved by the abstract theory of quasilinear parabolic systems. 
	
	\begin{lemma}\label{LS}
		Let $ \Omega \subset \mathbb{R}^2 $ be a bounded domain with smooth boundary and all the parameters be positive constants. Assume that $ u_0 \in W^{2,\infty}(\Omega) $, $ ( v_0, w_0)\in [W^{1,\infty}(\Omega)]^2 $ with $ u_0,v_0,w_0 \gneqq 0 $ and $ \sigma \ge 1 $. Then there exists a $ T_{max} \in (0,\infty] $ such that the problem \eqref{model} admits a unique classical solution $ (u,v,w)\in [C^0(\bar \Omega \times [0,T_{max})\cap C^{2,1}(\bar \Omega \times (0,T_{max})) ]^3 $
		satisfying $ u,v,w > 0 $ for all $ t>0 $. Moreover,
		\begin{equation*}
			if \ T_{max} < \infty,\ then \  \underset{t\nearrow T_{max}}{\mathrm{limsup}} \left( 	\| u(\cdot, t) \|_{W^{1,\infty}} +\| v(\cdot, t) \|_{W^{1,\infty}} + \| w(\cdot, t) \|_{L^\infty} \right) = \infty.
		\end{equation*}
	\end{lemma}
	\begin{proof}
		This proof is a slight adaptation of \cite[Lemma 2.1]{JWW}. That is, the 
		existence and uniqueness of classical solution $ (u,v,w) \in [C^0(\bar 
		\Omega \times [0,T_{max})\cap C^{2,1}(\bar \Omega \times (0,T_{max})) 
		]^3 $ can be obtained by Amann's theorem \cite[Theorem 7.3 and 
		Corollary 9.3]{HA1} and \cite{HA2}. Then applying the strong maximum 
		principle to the equations of \eqref{model}, we have $ (u,v,w)>0 $ for 
		all $ (x,t)\in \Omega\times (0, T_{max}) $ by means of the fact 
		$u_0,v_0,w_0 \gneqq 0  $.
	\end{proof}

	 The following lemma about the boundedness of prey can be regarded as a stepping stone to cope with such problems.
	
	\begin{lemma}\label{le-K1}
		Let the assumptions in Lemma \ref{LS} hold. Then the solution of \eqref{model} satisfies
		\begin{equation}\label{K1}
			\| u(\cdot, t) \|_{L^\infty} \leq K_1\quad \mathrm{for}\ \mathrm{all}\  t\in (0,T_{max}),
		\end{equation}
		where $ K_1 = \mathrm{max} \left\{{1},\| u_0 \|_{L^\infty}\right\} $.
	\end{lemma}
	\begin{proof}
		Using the first equation with Neumann boundary condition of \eqref{model}, we have
		\begin{equation*}
			\left\{
			\begin{aligned}
				&u_t -d_1\Delta v = r_1 u(1-u)-b_1uv-b_3uw\le r_1u(1-u) ,&x\in \ &\Omega, t>0,\\
				&\frac{\partial u}{\partial \nu}=0, &x\in \ &\partial \Omega, t>0,\\
				&u(x,0)=u_0(x), &x\in \ &\Omega.
			\end{aligned}
			\right.
		\end{equation*}
	Let $ u^*(t) $ be the solution of following ODE problem
	\begin{equation*}
		\left\{ 
		\begin{aligned}
			&\frac{d u^*(t) }{dt} = r_1 u^* (1-u^*),\quad t>0,\\
			& u^* (0) = \| u_0 \|_{L^\infty}.
		\end{aligned}
		\right.
		\end{equation*}
	Then, by solving this problem, one has $  u^*(t) \le K_1 := \max \left\{ 1, \| u_0 \|_{L^\infty} \right\} $. It is not difficult to see $ u^* (t) $ is a super-solution of the following PDE problem
		\begin{equation*}
		\left\{
		\begin{aligned}
			&U_t -d_1\Delta U = r_1U(1-U) ,&x\in \ &\Omega, t>0,\\
			&\frac{\partial U}{\partial \nu}=0, &x\in \ &\partial \Omega, t>0,\\
			&U(x,0)=u_0(x), &x\in \ &\Omega.
		\end{aligned}
		\right.
	\end{equation*}
It follows that $ 0< U(x,t) \le u^*(t)$. As a consequence, we have
\begin{equation*}
	0 < u(x,t)\le U(x,t)\le u^* (x,t) \le K_1
\end{equation*}
on the basis of comparison principle, which implies \eqref{K1}.
	\end{proof}
	
	Now, we establish basic $ L^1 $ estimate and spatio-temporal estimate about $ v $ and $ w $.
	\begin{lemma}\label{le-K2-K3}
		Let the assumptions in Lemma \ref{LS} hold and $ (u,v,w) $ be the solution of \eqref{model}, then one has two constants $ K_2, K_3>0 $ independent of $ t $ such that
		\begin{equation}\label{K2}
			\| v(\cdot, t) \|_{L^1} +	\| w(\cdot, t) \|_{L^1} \leq K_2\quad \mathrm{for}\ \mathrm{all}\ t\in (0,T_{max})
		\end{equation}
	and
	\begin{equation}\label{K3}
				\int_{t}^{t+\tau}\int_{\Omega} v^2(\cdot,s) + \int_{t}^{t+\tau}\int_{\Omega} w^{\sigma+1}(\cdot,s) \leq K_3\quad \mathrm{for}\ \mathrm{all}\ t\in (0,\widetilde{T}_{max}), 
		\end{equation}
	where
	\begin{equation}\label{tau}
		\tau:= \min \left\{ 1,\frac{1}{2}T_{max}  \right\}\ \mathrm{and}\ \widetilde{T}_{max}:=\left\{
		\begin{aligned}
			&T_{max}-\tau, & \mathrm{if}\ T_{max}<\infty,\\
			&\infty, & \mathrm{if}\ T_{max}=\infty.
		\end{aligned}
		\right.
	\end{equation}
	\end{lemma}
	\begin{proof}
		Using the equations of \eqref{model}, we have
		\begin{equation}\label{S2-1}
			\begin{aligned}
				&\frac{d}{dt}\int_{\Omega} u = r_1 \int_{\Omega} u(1-u) - b_1 \int_{\Omega}uv - b_3 \int_{\Omega} uw,\\
				&\frac{d}{dt}\int_{\Omega} v = r_2\int_{\Omega} v(1-v) + \int_{\Omega}uv -b_2\int_{\Omega} uw,\\
				&\frac{d}{dt} \int_{\Omega} w = r_3 \int_{\Omega} w(1-w^\sigma)+\int_{\Omega} vw +\int_{\Omega} uw.
			\end{aligned}
		\end{equation}
		Setting $ y_1 (t) := \left(1+\frac{b_1b_2}{b_3}\right) \int_{\Omega} u + b_1 \int_{\Omega} v + b_1b_2 \int_{\Omega} w $, simple algebraic calculations from \eqref{S2-1} show that 
		\begin{equation*}
			\begin{aligned}
				y_1'(t)+y_1(t)=& \left(r_1+1+\frac{b_1b_2}{b_3}\right)\int_{\Omega} u - r_1 \int_{\Omega} u^2 + (r_2+b_1) \int_{\Omega} v - r_2 \int_{\Omega} v^2 \\
				&+ (r_3+b_1b_2)\int_{\Omega} w -r_3 \int_{\Omega} w^{\sigma+1} - \frac{b_1^2b_2}{b_3} \int_{\Omega} uv\\
				\leq &- \frac{r_1}{2} \int_{\Omega} u^2  -\frac{r_2}{2} \int_{\Omega} v^2  -\frac{r_3}{2} \int_{\Omega} w^{\sigma+1} +c_1,
			\end{aligned}
		\end{equation*}
		where we have used Young's inequality. It follows that
		\begin{equation}\label{A1}
		y_1'(t)+y_1(t) + \delta_1 \left(\int_{\Omega}  u^2 +\int_{\Omega} v^2 + \int_{\Omega} w^{\sigma+1}\right) \le c_1, 
		\end{equation}
		where $ \delta_1 := \mathrm{min}\left\{ \frac{r_1}{2}, \frac{r_2}{2}, \frac{r_3}{2}\right\} $, which first gives \eqref{K2} by Gr\"{o}nwall's inequality and the positivity of solutions. By means of \eqref{K2}, we can integrate \eqref{A1} over $ (t,t+\tau) $ to get
		\begin{equation*}
			\delta_1\int_{t}^{t+\tau} \left(\int_{\Omega} v^2 + \int_{\Omega} w^{\sigma+1}\right) ds \le  \tau c_1+ y_1(t)-y_1(t+\tau) \le c_2.
		\end{equation*}
	The proof is complete.
	\end{proof}
	
	In order to make better use of spatio-temporal estimates like \eqref{K3} and Gr\"onwall type inequality below, we will 
	introduce the important lemma from \cite{W3} which shows the relation 
	between local-in-time estimate and exponential estimate (see another 
	estimation and proof  in \cite{SSW}).
	
	\begin{lemma}\label{inter}
		Let $ T>0 $, $ \tau \in (0,T) $, $ a>0 $ and $ b>0 $. Suppose that $ y: [0,T)\to [0,\infty) $ is absolutely continuous or  $ y\in C^0([0,T))\cap C^1((0,T)) $ and fulfills
		\begin{equation}\label{diff-ineq}
			y'(t) + ay(t) \le h(t)\quad \mathrm{for}\ \mathrm{all}\  t\in (0,T)
		\end{equation}
		with some nonnegative function $ h\in L_{loc}^1 (\mathbb{R}) $ satisfying $ \int_{t}^{t+\tau} h(s) \le b $ for all $ t\in [0,T-\tau) $. Then
		\begin{equation*}
			\int_{0}^{t} e^{-a(t-s)} h(s) ds \le \frac{b}{1-e^{-a\tau}}\quad \mathrm{for}\ \mathrm{all}\ t\in [0,T).
		\end{equation*}
		Consequently,
		\begin{equation}\label{y}
			y(t)\le e^{-at} y(0)+\frac{b}{1-e^{-a\tau}} \le y(0)+\frac{b}{1-e^{-a\tau}}\quad \mathrm{for}\ \mathrm{all}\ t\in [0,T).
		\end{equation}
	\end{lemma}
	\begin{proof}
		For fixed $ t\in (0,T) $, we pick a nonnegative integer $ N $ such that $ N\tau \le t < (N+1)\tau $, by nonnegativity of $ h $ we can estimate
		\begin{equation*}
			\begin{aligned}
				\int_{0}^{t} e^{-a(t-s)} h(s) ds &\le \int_{t-(N+1)\tau}^{t} e^{-a(t-s)} h(s) ds = \sum_{j=0}^{N} \int_{t-(j+1)\tau}^{t-j\tau} e^{-a(t-s)} h(s) ds\\
				&\le  \sum_{j=0}^{N} e^{-aj\tau} \int_{t-(j+1)\tau}^{t-j\tau} h(s) ds.
			\end{aligned}
		\end{equation*}
		Since the condition clearly entails that
		\begin{equation*}
			\int_{t-(j+1)\tau}^{t-j\tau} h(s) ds \le b\quad \mathrm{for}\ \mathrm{all} \ j\in \left\{0,\dots,N\right\}.
		\end{equation*}
		As a result, we have
		\begin{equation*}
			\int_{0}^{t} e^{-a(t-s)} h(s) ds \le b \sum_{j=0}^N e^{-aj\tau} \le b\sum_{j=0}^\infty e^{-aj\tau} = \frac{b}{1-e^{-a\tau}}.
		\end{equation*}
		This alongside with \eqref{diff-ineq} and ODE argument gives \eqref{y}, which completes the proof.
	\end{proof}
	We end this section by some well-known $ L^p-L^q $ estimates for the Neumann heat semigroup.
	\begin{lemma}[\cite{W}]\label{SG}
		Let $ (e^{td\Delta})_{t\geq 0} $ be the Neumann heat semigroup in $ \Omega $, and let $ \lambda_1 >0$ denote the first nonzero eigenvalue of $ -\Delta $ in $ \Omega $ under Neumann boundary conditions. Then for all $ t>0 $, there exist constants $ \gamma_i $$ (i=1,2,3,4) $ depending only on $ \Omega $ such that:
		
		(i) if $ 1\leq q \leq p \leq \infty $ then
		\begin{equation}\label{SG-1}
			\| e^{td\Delta}w \|_{L^p(\Omega)}\leq \gamma_1 \left( 1+t^{-\frac{N}{2}\left(\frac{1}{q}-\frac{1}{p}\right) }\right) \|w\|_{L^q(\Omega)}
		\end{equation}
		holds for all $ w\in L^q(\Omega) $.
		
		(ii) if $ 1\leq q \leq p \leq \infty $ then
		\begin{equation}\label{SG-2}
			\| \nabla e^{td\Delta}w \|_{L^p(\Omega)}\leq \gamma_2 \left( 1+t^{-\frac{1}{2}-\frac{N}{2}\left(\frac{1}{q}-\frac{1}{p}\right) }\right) e^{-\lambda_1d t}\|w\|_{L^q(\Omega)}
		\end{equation}
		holds for all $ w\in L^q(\Omega) $.
		
		(iii) if $ 2\leq p < \infty $ then
		\begin{equation}\label{SG-3}
			\| \nabla e^{td\Delta}w \|_{L^p(\Omega)}\leq \gamma_3 e^{-\lambda_1d t}\|\nabla w\|_{L^p(\Omega)}
		\end{equation}
		holds for all $ w\in W^{1,p}(\Omega) $.
		
		(iv) if $ 1 < q \leq p < \infty $ or $ 1<q<\infty $ and $ p=\infty $ then
		\begin{equation}\label{SG-4}
			\|  e^{td\Delta} \nabla \cdot w \|_{L^p(\Omega)}\leq \gamma_4 \left( 1+t^{-\frac{1}{2}-\frac{N}{2}\left(\frac{1}{q}-\frac{1}{p}\right) }\right) e^{-\lambda_1d t}\|w\|_{L^q(\Omega)}
		\end{equation}
		holds for all $ w\in \left(L^q\left(\Omega\right)\right)^N $.
	\end{lemma}
	
	\section{Global boundedness of solutions} \label{3}
	In this section, we will establish the global existence of classical solutions to \eqref{model}. To this end, suffice it to show that $ \| u(\cdot,t)\|_{W^{1,\infty}} $, $ \| v(\cdot,t)\|_{W^{1,\infty}}  $ and $ \| w (\cdot,t)\|_{L^\infty} $ are all bounded. We first analyze the equations of prey and primary predator to obtain the boundedness of $ \| v (\cdot,t)\|_{L^\infty} $. Then along the $ L^p $ estimates, $ \| \nabla u(\cdot,t)\|_{L^\infty} $, $ \| \nabla v(\cdot,t)\|_{L^\infty} $ can be derived in light of regularity theory and semigroup method.
	\subsection{Boundedness of $ \| v(\cdot,t)\|_{L^\infty} $} 
	This subsection is devoted to determining the upper bound of $ \| 
	v(\cdot,t)\|_{L^\infty} $. Let us start with two basic lemmas. The bounds 
	of $ \| \nabla u(\cdot,t)\|_{L^2} $ and $ \int_{t}^{t+\tau}\int_{\Omega} 
	|\Delta u(\cdot,s)|^2 $ can be obtained by standard testing procedure. Later the spatio-temporal estimate $  
	\int_{t}^{t+\tau}\int_{\Omega}\frac{|\nabla v(\cdot,s)|^2}{v(\cdot,s)} $ is as 
	a by-product from entropy estimate.
	\begin{lemma}\label{le-K4-K5}
		Let the assumptions in Lemma \ref{LS} hold and $ (u,v,w) $ be the solution of \eqref{model}. Then one has two positive constants $ K_4 $ and $ K_5 $ which are independent of $ t $ such that
		\begin{equation}\label{K4}
			\| \nabla u(\cdot, t) \|_{L^2} \leq K_4\quad \mathrm{for}\ \mathrm{all}\ t\in (0,{T}_{max})
		\end{equation}
	and
		\begin{equation}\label{K5}
			\int_{t}^{t+\tau}\int_{\Omega} |\Delta u(\cdot,s)|^2  \leq K_5\quad 
			\mathrm{for}\ \mathrm{all}\ t\in (0,\widetilde{T}_{max}),
		\end{equation}
		where  $ \tau $ and $ \widetilde{T}_{max} $ are defined in \eqref{tau}.
	\end{lemma}
	\begin{proof}
		Multiplying the first equation of \eqref{model} by $ -\Delta u $ and 
		integrating by parts, we observe that
		\begin{align*}
				\frac{1}{2}\frac{d}{dt} \int_{\Omega}|\nabla u|^2 +d_1 
				\int_{\Omega}|\Delta u|^2
				=&-r_1 \int_{\Omega} u(1-u)\Delta u +b_1 \int_{\Omega}uv \Delta 
				u + b_3 \int_{\Omega} uw \Delta u\\
				\le& r_1K_1(1+K_1) \int_{\Omega} |\Delta u| + b_1 K_1 
				\int_{\Omega} v |\Delta u| + b_3K_1 \int_{\Omega} w|\Delta u|\\
				 \le& \frac{d_1}{4} \int_{\Omega} |\Delta u|^2 + \frac{c_1}{2} 
				 \int_{\Omega} v^2 + \frac{c_2}{2} \int_{\Omega} w^2 
				 +\frac{c_3}{2},
		\end{align*}
	where Young's inequality and Lemma \ref{le-K1} are used. From which it follows that
	\begin{equation}\label{B1}
		\frac{d}{dt} \int_{\Omega} |\nabla u|^2 + \frac{3d_1}{2} \int_{\Omega} |\Delta u|^2 \le c_1 \int_{\Omega} v^2 + c_2 \int_{\Omega} w^2 + c_3.
	\end{equation}
Recall that $ \| u\|_{L^\infty} $ is bounded. By Young's inequality again,
\begin{equation*}
	\int_{\Omega} |\nabla u|^2 = \int_{\Omega} \nabla u \cdot \nabla u = - \int_{\Omega} u \Delta u \le \frac{d_1}{2} \int_{\Omega} |\Delta u|^2 + c_4.
\end{equation*}
Upon substituting this into \eqref{B1}, we have
\begin{equation}\label{B2}
	\frac{d}{dt} \int_{\Omega} |\nabla u|^2 +  \int_{\Omega} |\nabla u|^2 + \int_{\Omega} |\Delta u|^2 \le c_1 \int_{\Omega} v^2 + c_2 \int_{\Omega} w^2 + c_5.
\end{equation}
		Then applying Gr\"{o}nwall's inequality to \eqref{B2}, one has for all $ t\in (0,T_{max}) $, 
		\begin{equation*}
			\int_{\Omega} |\nabla u|^2\leq e^{-t}\| \nabla u_0\|_{L^2}+c_1\int_{0}^{t}e^{-(t-s)} \int_{\Omega}u^2 +c_2\int_{0}^{t}e^{-(t-s)} \int_{\Omega} w^{2} + c_5\int_{0}^{t}e^{-(t-s)},
		\end{equation*}
		which gives \eqref{K4} by employing $ u_0\in W^{1,\infty}(\Omega) $ and Lemma \ref{inter}. As a consequence, we integrate \eqref{B2} over $ (t,t+\tau) $ then utilize the analogous argument as in the proof of \eqref{K3} so as to obtain \eqref{K5} immediately.
	\end{proof}
	
	\begin{lemma}\label{le-K6}
		Let the assumptions in Lemma \ref{LS} hold and $ (u,v,w) $ be the solution of \eqref{model}. Then there exists a constant $ K_6>0 $ independent of $ t $ such that 
		\begin{equation}\label{K6}
			 \int_{t}^{t+\tau}\int_{\Omega}\frac{|\nabla 
			 v(\cdot,s)|^2}{v(\cdot,s)} \le K_6 \quad \mathrm{for}\ 
			 \mathrm{all}\ t\in (0,\widetilde{T}_{max}),
		\end{equation}
	where  $ \tau $ and $ \widetilde{T}_{max} $ are defined in \eqref{tau}.
	\end{lemma}
	\begin{proof}
		Noticing $ (v\ln v)_t = v_t\ln v -v_t $, $ v\ln v \ge -\frac{1}{e} $ and multiplying the second equation of \eqref{model} by $ \ln v $, we can obtain
		\begin{align}\label{B3}
			&\frac{d}{dt}\int_{\Omega} v\ln v + d_2 \int_{\Omega}\frac{|\nabla 
			v|^2}{v}\nonumber\\
			&=\xi  \int_{\Omega}\nabla u \cdot \nabla v + r_2\int_{\Omega}  
			v(1-v) \ln v + \int_{\Omega} (uv-b_2vw)\ln v\nonumber\\
			&\leq  \xi \int_{\Omega} v|\Delta u| +(r_2+K_1) \int_{\Omega} v\ln 
			v +\frac{K_2(r_2+b_2)}{e} \nonumber\\
			&\le c_1 \int_{\Omega} v^2 + c_2 \int_{\Omega} |\Delta 
			u|^2+(r_2+K_1) \int_{\Omega} |v\ln v| 
			+\frac{K_2(r_2+b_2)}{e},		
	    \end{align}
where Young's inequality, Lemma \ref{le-K1} and Lemma \ref{le-K2-K3} have been used. Adding $ \int_{\Omega} v\ln v $ at both sides of \eqref{B3} and combining with the fact $ |v\ln v| \le v^{\frac{3}{2}}+c_3 $, one has
\begin{equation*}
		\frac{d}{dt} \int_{\Omega} v\ln v +  \int_{\Omega} v\ln v +d_2 \int_{\Omega}\frac{|\nabla v|^2}{v} \le c_1 \int_{\Omega} v^2 + c_2 \int_{\Omega} |\Delta u|^2+c_4.
\end{equation*}
With the help of \eqref{K3} and \eqref{K5}, we can first get $ \| v\ln v\|_{L^1}\le c_6 $. Then it is evident that \eqref{K6} is right.
	\end{proof}

	We will now discuss two $ L^p $ estimates with respect to $ v $ and $ \nabla u $ which play a fundamental role in our further proof on account of the flexible application of semigroup theory and energy estimates.
	\begin{lemma}\label{le-K7}
		Let the assumptions in Lemma \ref{LS} hold, then the solution of \eqref{model} satisfies
		\begin{equation}\label{K7}
			\| v(\cdot,t) \|_{L^p} \leq K_7(p)\quad \mathrm{for}\ \mathrm{all}\ t\in (0,{T}_{max}),
		\end{equation}
		where $ p\in [1,\infty) $ and $ K_7>0 $ is a constant independent of $ t $.
	\end{lemma}
	\begin{proof}
		Firstly, multiplying the first equation of \eqref{model} by $ v^{p-1}$ with $ p>2 $ and integrating it by parts, we have
		\begin{align}\label{B4}
				&\frac{1}{p}\frac{d}{dt}\int_{\Omega} v^p 
				+d_2(p-1)\int_{\Omega} v^{p-2}|\nabla u|^2 \nonumber\\
				&=\xi(p-1)\int_{\Omega}v^{p-1}\nabla u\cdot \nabla v + 
				\int_{\Omega} v^p(r_2-r_2v+u-b_2w)\nonumber\\
				&\leq \frac{\xi(p-1)}{p}\int_{\Omega} \nabla (v^p)\cdot \nabla 
				u +(r_2+K_1)\int_{\Omega} v^p- r_2 
				\int_{\Omega}v^{p+1}\nonumber\\
				&\leq \frac{\xi(p-1)}{p}\int_{\Omega} |v^p \Delta u| 
				+(r_2+K_1)\int_{\Omega} v^p.
		\end{align}
		With the help of the Gagliardo-Nirenberg inequality $ \| v^{\frac{p}{2}} \|_{L^4}^2\leq c_1 \| \nabla v^{\frac{p}{2}} \|_{L^2}\| v^{\frac{p}{2}} \|_{L^2} +c_2\| v^{\frac{p}{2}} \|_{L^2}^2 $ and the H\"{o}lder's inequality, we derive that
		\begin{equation}\label{B5}
			\|v^p\Delta u \|_{L^1}\leq \| \Delta u\| _{L^2}\| v^{\frac{p}{2}} \|_{L^4}^2 \leq \| \Delta u \|_{L^2}(c_1\| \nabla v^{\frac{p}{2}} \|_{L^2}\| v^{\frac{p}{2}} \|_{L^2}+c_2\| v^{\frac{p}{2}}\|_{L^2}^2).
		\end{equation}
		Then substituting \eqref{B5} into \eqref{B4}, one has
		\begin{equation*}
			\begin{aligned}
				&\frac{1}{p}\frac{d}{dt}\| v \|_{L^p}^p +\frac{4d_2(p-1)}{p^2}\| \nabla u^{\frac{p}{2}} \| _{L^2}^2\\
				&\leq \frac{\xi(p-1)}{p}\| \Delta u \|_{L^2}(c_1\| \nabla v^{\frac{p}{2}} \|_{L^2}\| v^{\frac{p}{2}} \|_{L^2}+c_2\| v^{\frac{p}{2}}\|_{L^2}^2)+(r_2+K_1)\| v^{\frac{p}{2}} \|_{L^2}^2\\
				&\leq \frac{2d_2(p-1)}{p^2}\|\nabla v^{\frac{p}{2}}\|_{L^2}^2+c_3\| \Delta u\|_{L^2}^2\| v^{\frac{p}{2}} \|_{L^2}^2+c_4\| \Delta v \|_{L^2}\|u^{\frac{p}{2}}\|_{L^2}^2+(r_2+K_1)\| v^{\frac{p}{2}} \|_{L^2}^2\\
				&\leq \frac{2d_1(p-1)}{p^2}\|\nabla u^{\frac{p}{2}}\|_{L^2}^2+c_5\left( \| \Delta u\|_{L^2}^2+1 \right) \|u^{\frac{p}{2}} \|_{L^2}^2,
			\end{aligned}
		\end{equation*}
		which gives
		\begin{equation}\label{B6}
			\frac{d}{dt} \| v \|_{L^p}^p \leq c_6\left(\| \Delta u \|_{L^2}^2+1  \right) \| v \|_{L^p}^p.
		\end{equation}
		
		In view of \eqref{B6}, \eqref{K7} will be established based on spatio-temporal estimates. Employing \eqref{K2}, \eqref{K6} and Sobolev embedding theorem $ W^{1,2}\hookrightarrow L^{q} $ which is feasible for all $ q\in[1,\infty) $ when $ N=2 $, we have
		\begin{equation}\label{B7}
			\begin{aligned}
				\int_{t}^{t+\tau} \|v \|_{L^p} \le \int_{t}^{t+\tau} \| \sqrt{v}\|_{L^{2p}}^2
				&\leq c_7 \int_{t}^{t+\tau} \left( \| \nabla \sqrt{v}\|_{L^2}^2+\| \sqrt{v}\|_{L^2}^2 \right)\\
				&= c_7 \int_{t}^{t+\tau}\int_{\Omega} \frac{|\nabla v|^2}{v}+c_7\int_{t}^{t+\tau}\int_{\Omega}v\le c_8.
			\end{aligned}
		\end{equation}
		Thus, for any $ t\in (0,T_{max}) $, it is plain to find $ t_\star = t_\star(t)\in ((t-\tau)_{+},t) $ such that $ t_\star\geq 0 $ and 
		\begin{equation}\label{B8}
			\int_{\Omega} v^p(\cdot,t_\star) \leq c_9
		\end{equation}
		on the basis of \eqref{B7}. In addition, \eqref{K5} ensures that
		\begin{equation}\label{B9}
			\int_{t_\star}^{t_\star +\tau}\int_{\Omega} |\Delta u (\cdot,t,s)|^2 \leq 
			c_{10}\quad \mathrm{for}\ \mathrm{all}\ t_\star\in 
			(0,\widetilde{T}_{max}).
		\end{equation}
		
		Finally, applying Gr\"{o}nwall inequality to \eqref{B6} over $ (t_\star, t) $ and then utilizing \eqref{B8} and \eqref{B9}, one has 
		\begin{equation*}
			\int_{\Omega}v^p(\cdot,t) \leq \int_{\Omega} v^p(\cdot,t_\star) \cdot e^{{c_6}\int_{t_\star}^{t}\|\Delta u(\cdot, s)\|_{L^2}^2+1 ds }\leq c_9 e^{c_6\left(c_{10}+1\right)}
		\end{equation*}
		for all $ t\in (0,T_{max}) $, which completes the proof.
	\end{proof}
	\begin{lemma}\label{le-K8}
		Let the assumptions in Lemma \ref{LS} hold. Then for any given $ \sigma \ge 1 $, the solution of \eqref{model} satisfies
		\begin{equation}\label{K8}
			\| \nabla u(\cdot,t) \|_{L^p}\leq K_8(p),\quad \left\{
			\begin{aligned}
				&1\le p<\frac{2\sigma +2}{3-\sigma}, &\mathrm{if}\ &1 \le \sigma <3,\\
				&1\le p < \infty, &\mathrm{if}\ & \sigma \ge 3
			\end{aligned}
			\right.
		\end{equation}
		for all $ t\in (0,T_{max}) $, where $ K_8>0 $ is a constant independent of $ t $.
	\end{lemma}
	\begin{proof}
		The proof is similar as the one in \cite[Lemma 3.3]{LW} based upon 
		delicate semigroup estimates and therefore we omit the detailed information for 
		conciseness. 
	\end{proof}
	Now, we are able to establish the boundedness of $ \| v(\cdot,t) \|_{L^\infty} $ by virtue of the powerful tools Lemma \ref{le-K7} and Lemma \ref{le-K8}.
	\begin{lemma}\label{le-K9}
		Suppose $ \sigma >1 $ and $ (u,v,w) $ is the local classical solution 
		of \eqref{model}. Then one has a constant $ K_{9}>0 $ independent of $ 
		t $ and $ \chi $ such that
		\begin{equation}\label{K9}
			\| v(\cdot, t)\|_{L^\infty} \leq K_{9}\quad \mathrm{for}\ \mathrm{all}\ t\in(0,T_{max}).
		\end{equation}
	\end{lemma}
	\begin{proof}
		Due to $ 
		\frac{2\sigma+2}{3-\sigma}>2 $ for $ \sigma>1 $, we remark that there exists a constant $ \tilde{p}>2 $ such that $ 
		\|\nabla u\|_{L^{\tilde{p}}}\le K_{8}(\tilde{p}) $ for given $ 
		\sigma>1 $ by Lemma \ref{le-K8}.
		We rewrite the first equation of \eqref{model} as follows:
		\begin{equation}\label{C1}
			v_t-d_2\Delta v+v = -\nabla \cdot (\xi v\nabla u) + r_2v(1-v)+v+uv-b_2vw.
		\end{equation}
		Applying the variation-of-constants formula to \eqref{C1}, one has
		\begin{equation*}
			\begin{aligned}
				v(\cdot,t) =& e^{t(d_2\Delta-1)} v_0 - \xi \int_{0}^{t} e^{(t-s)(d_2\Delta-1)}\nabla \cdot (v\nabla u) ds\\
				& + r_2\int_{0}^{t}e^{(t-s)(d_2\Delta -1)}v(1-v) ds +  \int_{0}^{t} e^{(t-s)(d_2\Delta -1)}v(1+u-b_2w)ds \\
				\leq & e^{t(d_1\Delta-1)} u_0 + \xi \int_{0}^{t} e^{(t-s)(d_1\Delta-1)}\nabla \cdot (v\nabla u) ds\\
				&+ (r_2+K_1+1)\int_{0}^{t}e^{(t-s)(d_2\Delta -1)}v ds.
			\end{aligned}
		\end{equation*}
		Then, through the standard $ L^p-L^q $ estimates \eqref{SG-1} and \eqref{SG-4} in Lemma \ref{SG} with $ 2<q<\tilde{p} $, we observe that
		\begin{align}\label{C2}
			\| v (\cdot,t)\|_{L^\infty} \leq & \| e^{t(d_2\Delta-1)} v_0 \|_{L^\infty} + \xi \int_{0}^{t} \| e^{(t-s)(d_2\Delta-1)}\nabla \cdot (v\nabla u)\|_{L^\infty} ds \nonumber\\
			& + (r_2+K_1+1)\int_{0}^{t} \| e^{(t-s)(d_2\Delta -1)}v \| _{L^\infty}ds\nonumber\\
			\leq & \| v_0 \|_{L^\infty} + \gamma_4 \xi \int_{0}^{t} \left(1+ (t-s)^{-\frac{1}{2}-\frac{1}{q}}\right)e^{-(\lambda_1 d_2 +1)(t-s)}\| v \nabla u \|_{L^q} ds\nonumber\\
			&+ \gamma_1 (r_2+K_1+1)\int_{0}^{t} \left( 1+ (t-s)^{-\frac{1}{3}}\right) e^{-(t-s)} \| v \|_{L^3}ds .
		\end{align}
		The Young's inequality entails that
		\begin{equation*}
			\| v \nabla u \|_{L^q}\leq \left( \int_{\Omega}  v^{\frac{\tilde{p}q}{\tilde{p}-q}}+\int_{\Omega}|\nabla u|^{\tilde{p}}\right) ^{\frac{1}{q}},
		\end{equation*}
		which in combination with \eqref{K7} and \eqref{K8} leads to
		\begin{equation}
			\gamma_4 \xi \int_{0}^{t} \left(1+ (t-s)^{-\frac{1}{2}-\frac{1}{q}}\right)e^{-(\lambda_1 d_2 +1)(t-s)}\| v \nabla u \|_{L^q} ds \leq c_1 \left(\Gamma\left(\frac{1}{2}-\frac{1}{q}\right)+1\right).
		\end{equation}
		Moreover, since $ \| v (\cdot,t) \|_{L^3} $ is bounded from \eqref{K7} we have
		\begin{equation*}
			\gamma_1 (r_2+K_1+1)\int_{0}^{t} \left( 1+ (t-s)^{-\frac{1}{3}}\right) e^{-(t-s)} \| v \|_{L^3}ds 
			\leq c_2\left( \Gamma\left(\frac{2}{3}\right)+1\right).
		\end{equation*}

		In conclusion, \eqref{K9} is directly obtained from substituting above two estimates into \eqref{C2}, which completes the proof.
	\end{proof}
	
	\subsection{Boundedness of $ \| w(\cdot,t)\|_{L^\infty} $}
	
	In this subsection, we will first show the boundedness of $ \| \nabla v(\cdot,t)\|_{L^2} $ which immediately unfolds the further proof.
	\begin{lemma}\label{le-K10-K11}
		Let $ \sigma>1 $ and $ (u,v,w) $ be the local classical solution of \eqref{model}. Then there exists a constant $ K_{10}>0 $ independent of $ t $ such that 
		\begin{equation}\label{K10}
		 \| \nabla v(\cdot,t) \|_{L^2} \le K_{10} \quad \mathrm{for}\ \mathrm{all}\  t\in (0,T_{max})
		\end{equation}
	and
	\begin{equation}\label{K11}
		\int_{t}^{t+\tau}\int_{\Omega} |\Delta v(\cdot,s)|^2  \leq K_{11}\quad 
		\mathrm{for}\ \mathrm{all}\ t\in (0,\widetilde{T}_{max}),
	\end{equation}
	where  $ \tau $ and $ \widetilde{T}_{max} $ are defined in \eqref{tau}.
	\end{lemma}
	\begin{proof}
	Multiplying the second equation of system \eqref{model} by $ - \Delta v $ 
	and integrating it by parts, one has
	\begin{equation*}
		\begin{aligned}
			&\frac{1}{2}\frac{d}{dt} \int_{\Omega} |\nabla v|^2 +d_2 \int_{\Omega} |\Delta v|^2\\
			&= \int_{\Omega} \nabla \cdot (\xi v\nabla u)\Delta v-r_2 \int_{\Omega} v(1-v)\Delta v - \int_{\Omega} uv\Delta v + b_2 \int_{\Omega} vw \Delta v\\
			&\le \int_{\Omega} \nabla \cdot (\xi v\nabla u)\Delta v+K_9[r_2(1 +K_9)+K_1]\int_{\Omega} |\Delta v|+b_2K_9\int_{\Omega} w|\Delta v|\\
			& \le \xi \int_{\Omega} |\nabla v| | \nabla u| |\Delta v| + \xi K_9\int_{\Omega} |\Delta u||\Delta v|+ c_1 \int_{\Omega} |\Delta v|+c_2 \int_{\Omega}w|\Delta v|\\
			&  \le \xi \int_{\Omega} |\nabla v| | \nabla u| |\Delta v| + \frac{d_2}{4}\int_{\Omega} |\Delta v|^2 + c_3 \int_{\Omega} |\Delta u|^2 + c_4 \int_{\Omega} w^2 +c_5,
		\end{aligned}		
	\end{equation*}
	where we have used Young's inequality and the boundedness of $ \| u\|_{L^\infty} $ and $ \| v\|_{L^\infty} $. 
	
	It remains to deal with the first integral term on the right side of the 
	above inequality. Using the same argument as in Lemma \ref{le-K9}, we 
	choose $ \tilde{p} >2 $ such that $ \| \nabla u\|_{L^p} \le 
	K_8(\tilde{p})  $. Then employing H\"older's inequality, we can derive
	\begin{equation}\label{C3}
		\begin{aligned}
			\xi \int_{\Omega} |\nabla v| | \nabla u| |\Delta v| & \le \xi  \left( \int_{\Omega} |\nabla u|^{\tilde{p}} \right)^{\frac{1}{\tilde{p}}}\left( \int_{\Omega} |\Delta v|^{\frac{\tilde{p}}{\tilde{p}-1}}|\nabla v|^{\frac{\tilde{p}}{\tilde{p}-1}}\right)^{\frac{\tilde{p}-1}{\tilde{p}}}\\
			& \le \xi K_8 \left(\int_{\Omega} |\Delta v|^2\right)^{\frac{1}{2}}\left( \int_{\Omega} |\nabla v|^{\frac{\tilde{p}q}{p-1}} \right)^{\frac{\tilde{p}-1}{\tilde{p}q}},
		\end{aligned}
	\end{equation}
where $ q $ is the conjugate exponent of $ \frac{2(\tilde{p}-1)}{\tilde{p}} $ satisfying $ \frac{\tilde{p}}{2(\tilde{p}-1)}+\frac{1}{q}=1 $. This identity naturally implies
\begin{equation*}
	0<\alpha : = 1-\frac{2(\tilde{p}-1)}{\tilde{p}q}<1.
\end{equation*}
By means of Gagliardo-Nirenberg inequality, one has
 \begin{equation}\label{C4}
 	\|  \nabla v \|_{L^{\frac{\tilde{p}q}{\tilde{p}-1}}} \le c_6 \| \Delta v\|_{L^2} ^{\alpha} \| \nabla v \|_{L^2}^{1-\alpha} +c_7 \| \nabla v \|_{L^2}.
 \end{equation}
 It follows, substituting \eqref{C4} into \eqref{C3}, that
 \begin{align*}
 		\xi \| |\nabla v||\nabla u||\Delta v|\|_{L^1}&\le \xi K_8 \| \Delta v \|_{L^2}(c_6 \| \Delta v\|_{L^2} ^{\alpha} \| \nabla v \|_{L^2}^{1-\alpha} + c_7\| \nabla v \|_{L^2})\\
 		& \le \frac{d_2}{12} \| \Delta v \|_{L^2}^2 + \xi K_8 (c_8+c_7) \| 
 		\Delta v \|_{L^2}\| \nabla v \|_{L^2}\\
 		& \le \frac{d_2}{6} \| \Delta v \|_{L^2}^2  + c_{9} \| \nabla v 
 		\|_{L^2}^2,
 \end{align*}
 which in combination with $ 	\left(c_{9} + \frac{1}{2}\right) \| \nabla v 
 \|_{L^2}^2 \le  \frac{d_2}{12}\| \Delta v \|_{L^2}^2 + c_{10} $ entails the 
 differential inequality
\begin{equation*}
 \frac{d}{dt} \int_{\Omega} |\nabla v|^2 + \int_{\Omega} |\nabla v|^2 + d_2 
 \int_{\Omega} |\Delta v|^2 \le c_{11} \left( 1+\int_{\Omega} |\Delta u|^2 + 
 \int_{\Omega} w^2 \right).
\end{equation*}

Finally, we can manipulate Lemma \ref{inter} with \eqref{K3} and \eqref{K5} to obtain \eqref{K10}, as well as \eqref{K11} by the same argument as in the proof of Lemma \ref{le-K2-K3}.
\end{proof}

With the help of Lemma \ref{le-K10-K11}, we can establish the following 
spatio-temporal estimate.

	\begin{lemma}\label{le-K12}
	Let $ \sigma>1 $ and $ (u,v,w) $ be the local solution of \eqref{model} obtained in Lemma \ref{LS}, one can find a constant $ 
	K_{12}>0 $ independent of $ t $ such that 
	\begin{align}\label{K12}
		\int_{t}^{t+\tau}\int_{\Omega}\frac{|\nabla 
		w(\cdot,s)|^2}{w(\cdot,s)}\le K_{12} \quad \mathrm{for}\ 
		\mathrm{all}\   t\in(0,\widetilde{T}_{max}),
	\end{align}
	where  $ \tau $ and $ \widetilde{T}_{max} $ are defined in \eqref{tau}.
\end{lemma}

\begin{proof}
	Multiplying the third equation of \eqref{model} and $ \ln w+1 $ together, we can obtain that
	\begin{align}\label{111}
			&\frac{d}{dt}\int_{\Omega}w\ln w+\int_{\Omega}\frac{|\nabla 
				w|^2}{w}\nonumber\\
			&=\chi\int_{\Omega}\nabla 
			w\cdot\nabla(uv)+\int_{\Omega}\left(r_3w\left(1-w^\sigma\right)+w(u+v)\right)
			(\ln 
			w+1)\nonumber\\
			&=: I_1+I_2.
	\end{align}
	Using the boundedness of $ \|\nabla 
	u\|_{L^2} $ and $ \|\nabla v\|_{L^2} $, we can derive by Gagliardo-Nirenberg's 
	inequality that
	\begin{align}\label{lsz11}
		\|\nabla u\|_{L^4}^4\le c_1\left(\|\Delta u\|_{L^2}^{\frac{1}{2}}\|\nabla 
		u\|_{L^2}^{\frac{1}{2}}+\|\nabla u\|_{L^2}\right)^4\le c_2\|\Delta 
		u\|_{L^2}^2+c_2, 
	\end{align}
	similarly,
	\begin{align}\label{lsz22}
		\|\nabla v\|_{L^4}^4\le c_3\|\Delta 
		v\|_{L^2}^2+c_3, 
	\end{align}
	these along with \eqref{K1} and \eqref{K9} yield
	\begin{align}\label{222}
		I_1=&-\chi\int_{\Omega}w\Delta(uv)=-\chi\int_{\Omega}w\nabla\cdot(u\nabla 
		v+v\nabla u)\nonumber\\
		=& -2\chi\int_{\Omega}w\nabla u\cdot\nabla v-\chi\int_{\Omega}wu\Delta 
		v-\chi\int_{\Omega}wv\Delta u\nonumber\\
		\le& 2\chi\int_{\Omega}w|\nabla u||\nabla 
		v|+\chi K_1\int_{\Omega}w|\Delta 
		v|+\chi K_9\int_{\Omega}w|\Delta 
		u|\nonumber\\
		\le& \frac{r_3}{4}\int_{\Omega}w^{\sigma+1}+c_4\int_{\Omega}|\nabla 
		u|^4+c_4\int_{\Omega}|\nabla v|^4+c_4\int_{\Omega}|\Delta 
		u|^2+c_4\int_{\Omega}|\Delta v|^2+c_4\nonumber\\
		\le& \frac{r_3}{4}\int_{\Omega}w^{\sigma+1}+c_5\int_{\Omega}|\Delta 
		u|^2+c_5\int_{\Omega}|\Delta v|^2+c_5.
	\end{align}
	
	On the other hand, noting the fact $ w\ln w\ge -\frac{1}{e} $ and $ |w\ln w|\le 
	w^\frac{3}{2}+c $, one can obtain
	\begin{align}\label{333}
		I_2+\int_{\Omega}w\ln w\le& 
		-r_3\int_{\Omega}w^{\sigma+1}+\frac{r_3}{e}\int_{\Omega}w^\sigma+(r_3+K_1+K_9)\int_{\Omega}w\nonumber\\
		&+(r_3+1+K_1+K_9)\int_{\Omega}|w\ln w|\nonumber\\
		\le& 
		-\frac{3r_3}{4}\int_{\Omega}w^{\sigma+1}+c_6\int_{\Omega}w^{\frac{3}{2}}+
		c_6\nonumber\\
		\le& -\frac{3r_3}{4}\int_{\Omega}w^{\sigma+1}+c_7\|\nabla 
		w^{\frac{1}{2}}\|_{L^2}\|w^{\frac{1}{2}}\|_{L^2}^2+c_7\|w^\frac{1}{2}\|
		_{L^2}^3+c_6\nonumber\\
		\le& 
		-\frac{3r_3}{4}\int_{\Omega}w^{\sigma+1}+\frac{1}{2}\int_{\Omega}\frac{|\nabla
			w|^2}{w}+c_8.
	\end{align}
	Substituting \eqref{222} and \eqref{333} into \eqref{111}, one has
	\begin{align}
		\frac{d}{dt}\int_{\Omega}w\ln w+\int_{\Omega}w\ln w+ 
		\frac{1}{2}\int_{\Omega}\frac{|\nabla
			w|^2}{w}\le 
		c_5\int_{\Omega}|\Delta 
		u|^2+c_5\int_{\Omega}|\Delta v|^2+c_9.
	\end{align}

In conclusion, employing \eqref{K5}, \eqref{K11} and Lemma \ref{inter} again, we can complete this proof.
\end{proof}

Next, the $ L^p $ bound of $ w $ will be obtained.

\begin{lemma}\label{le-K13}
	Assume $ \sigma>1 $ and $ (u,v,w) $ is the local solution of \eqref{model}, then one has
	\begin{align}\label{K13}
		\|w(\cdot,t)\|_{L^p}\le K_{13}(p) \quad \mathrm{for}\ \mathrm{all}\  t\in(0,T_{max}),
	\end{align}
where $ p\in [1, \infty) $ and $ K_{13} $ is a positive constant independent of $ t $.
\end{lemma}
\begin{proof}
	Using the third equation of \eqref{model} and noting the boundedness of $ 
	\|u\|_{L^\infty} $ and $ \|v\|_{L^\infty} $, one has that for all $ p>2 $
	\begin{align*}
		\frac{1}{p}\frac{d}{dt}\int_{\Omega}w^p=& 
		\int_{\Omega}w^{p-1}w_t\nonumber\\
		=& \int_{\Omega}w^{p-1}[\Delta w - \nabla \cdot (\chi w\nabla (uv))+r_3 
		w(1-w^\sigma)+vw+uw]\nonumber\\
		\le& -(p-1)\int_{\Omega}w^{p-2}|\nabla 
		w|^2-\frac{\chi(p-1)}{p}\int_{\Omega}w^p\Delta(uv)-\frac{r_3}{2}\int_{\Omega}
		w^{p+\sigma}+c_1,
	\end{align*}
	which means
	\begin{align}\label{lll1}
		&\frac{d}{dt}\int_{\Omega}w^p+ p(p-1)\int_{\Omega}w^{p-2}|\nabla w|^2+ 
		\frac{pr_3}{2}\int_{\Omega}
		w^{p+\sigma}\nonumber\\
		\le& -\chi(p-1)\int_{\Omega}w^p(u\Delta v+v\Delta u+2\nabla u\cdot\nabla 
		v)\nonumber\\
		\le& c_2\int_{\Omega}w^p|\Delta u|+c_2\int_{\Omega}w^p|\Delta 
		v|+c_2\int_{\Omega}w^p|\nabla u||\nabla v|.
	\end{align}
	With the help of the Gagliardo-Nirenberg inequality 
	\begin{align*}
		\|w^{\frac{p}{2}}\|_{L^4}\le c_3\left(\|\nabla 
		w^{\frac{p}{2}}\|_{L^2}^\frac{1}{2}\|w^{\frac{p}{2}}\|_{L^2}^\frac{1}{2}
		+\|w^{\frac{p}{2}}\|_{L^2}\right),
	\end{align*}
	we can derive
	\begin{align}\label{lll2}
		c_2\|w^p\Delta u\|_{L^1}\le& c_2\|\Delta 
		u\|_{L^2}\|w^{\frac{p}{2}}\|_{L^4}^2\nonumber\\
		\le& c_4\|\Delta u\|_{L^2}\left(\|\nabla 
		w^{\frac{p}{2}}\|_{L^2}\|w^{\frac{p}{2}}\|_{L^2}
		+\|w^{\frac{p}{2}}\|_{L^2}^2\right)\nonumber\\
		\le& \frac{p-1}{p}\|\nabla 
		w^{\frac{p}{2}}\|_{L^2}^2+c_5\left(\|\Delta 
		u\|_{L^2}^2+1\right)\|w^{\frac{p}{2}}\|_{L^2}^2,
	\end{align}
	similarly,
	\begin{align}\label{lll3}
		c_2\|w^p\Delta v\|_{L^1}\le \frac{p-1}{p}\|\nabla 
		w^{\frac{p}{2}}\|_{L^2}^2+c_6\left(\|\Delta 
		v\|_{L^2}^2+1\right)\|w^{\frac{p}{2}}\|_{L^2}^2.
	\end{align}
	Noting \eqref{lsz11} and \eqref{lsz22}, one can obtain
	\begin{align}\label{lll4}
		c_2\int_{\Omega}w^p|\nabla u||\nabla v|\le & c_2\|w^{\frac{p}{2}}\|_{L^4}^2
		\|\nabla u\|_{L^4}\|\nabla v\|_{L^4}\nonumber\\
		\le& \frac{c_2}{2}\|w^{\frac{p}{2}}\|_{L^4}^2\left(\|\nabla 
		u\|_{L^4}^2+\|\nabla v\|_{L^4}^2\right)\nonumber\\
		\le& c_7\left(\|\nabla 
		w^{\frac{p}{2}}\|_{L^2}\|w^{\frac{p}{2}}\|_{L^2}
		+\|w^{\frac{p}{2}}\|_{L^2}^2\right)\left(\|\Delta u\|_{L^2}+\|\Delta 
		v\|_{L^2}+1\right)\nonumber\\
		\le& \frac{2(p-1)}{p}\|\nabla 
		w^{\frac{p}{2}}\|_{L^2}^2+c_8\left(\|\Delta u\|_{L^2}^2+\|\Delta 
		v\|_{L^2}^2+1\right)\|w^{\frac{p}{2}}\|_{L^2}^2.
	\end{align}

	Substituting \eqref{lll2} \eqref{lll3} and \eqref{lll4} into \eqref{lll1}, we 
	get
	\begin{align}\label{lll5}
		\frac{d}{dt}\|w\|_{L^p}^p\le c_9\left(\|\Delta u\|_{L^2}^2+\|\Delta 
		v\|_{L^2}^2+1\right)\|w\|_{L^p}^p
	\end{align}
and reach \eqref{K13} by analogous method of Lemma \ref{le-K7}.
	Noticing
	\begin{align*}
		\int_{t}^{t+\tau} \|w \|_{L^p}& \le \int_{t}^{t+\tau} \| 
		\sqrt{w}\|_{L^{2p}}^2\nonumber\\
		&\leq c_{10} \int_{t}^{t+\tau} \left( \| \nabla \sqrt{w}\|_{L^2}^2+\| 
		\sqrt{w}\|_{L^2}^2 \right)\nonumber\\
		&= c_{10} \int_{t}^{t+\tau}\int_{\Omega} \frac{|\nabla 
			w|^2}{w}+c_{10}\int_{t}^{t+\tau}\int_{\Omega}w\nonumber\\
		&\leq c_{11},
	\end{align*}
	where we have used \eqref{K2} and \eqref{K12}. Thus, for any $ t\in 
	(0,T_{max}) $, we can find $ t^\star = t^\star(t)\in ((t-\tau)_{+},t) $ such that $ 
	t^\star\geq 0 $ and
	\begin{equation}\label{lll6}
		\int_{\Omega} w^p(\cdot,t^\star) \leq c_{12}.
	\end{equation}
	
	Finally, applying Gr\"{o}nwall inequality to \eqref{lll5} over $ (t^\star, t) $ and
	then utilizing \eqref{lll6}, one has
	\begin{equation*}
		\int_{\Omega}w^p(\cdot,t) \leq \int_{\Omega} w^p(\cdot,t^\star) \cdot 
		e^{{c_9}\int_{t^\star}^{t}\left(\|\Delta v(\cdot, s)\|_{L^2}^2+\|\Delta 
			v(\cdot, s)\|_{L^2}^2+1\right) ds }\leq c_{12} e^{c_6\left(c_{13}+1\right)}.
	\end{equation*}
	The proof of this lemma is completed.
\end{proof}

Now, we can derive the boundedness of $ \| \nabla u(\cdot,t)\|_{L^\infty} $, $ \| \nabla v(\cdot,t)\|_{L^\infty} $ and $ \| w(\cdot,t)\|_{L^\infty} $ in next two lemmas.

	\begin{lemma}\label{le-K14}
	Suppose $ \sigma>1$ and $ (u,v,w) $ is the local classical solution of \eqref{model}. Then one has
	\begin{equation}\label{K14}
		\| \nabla u(\cdot,t)\|_{L^\infty} +\| \nabla v(\cdot,t) \|_{L^\infty}\le K_{14}\quad \mathrm{for}\ \mathrm{all}\ t\in(0,T_{max}),
	\end{equation}
	where $ K_{14}>0 $ is a constant independent of $ t $.
\end{lemma}
\begin{proof}
	First of all, we introduce the standard parabolic regularity theory (see \cite{HM,KS}): Let $ T\in (0,\infty] $ and suppose $ z\in C^0(\bar \Omega \times [0,T)\cap C^{2,1}(\bar \Omega \times (0,T)) $ is a solution of
	\begin{equation*}
		\left\{
		\begin{aligned}
			&z_t = \Delta z-z+g, & x&\in \Omega, t\in (0,T),\\
			&\frac{\partial z}{\partial \nu} =0, & x&\in \partial \Omega, t\in (0,T),
		\end{aligned}
		\right.
	\end{equation*}
where $ g(t) \in L^p(\Omega) $ is bounded for all $ t\in (0,T) $. Then there exists a constant $ C>0 $ such that

	\begin{equation*}
	\|  z(\cdot,t) \|_{W^{1,r}} \leq C,\quad  r\in \left\{
	\begin{aligned}
		&[1,\frac{Np}{N-p}),  &\mathrm{if}\ &p\le N,\\
		&[1,\infty], &\mathrm{if}\ & p>N.
	\end{aligned}
	\right.
\end{equation*}

	With this result in hand, we can rewrite the first equation of \eqref{model} as $ u_t =: d_1\Delta u -u +F(u,v,w) $, where $ F(u,v,w) = u(r_1+1-r_1u-b_1v-b_3w) $. Noting $ \| u\|_{L^\infty} $ and $ \| v\|_{L^\infty} $ are bounded, $ \| w\|_{L^3} $ is bounded as well according to Lemma \ref{le-K13}. Hence we can derive the boundedness of $ \| \nabla u \|_{L^\infty} $. 
	
	Next, we need an auxiliary regularity result before we conduct the estimate of $ \| \nabla u\|_{L^\infty}$. Using Lemma 2.4 in \cite{CJ}, $ u_0\in W^{2,\infty}(\Omega) $ and the fact that $ \| w\|_{L^6} $ has the uniform-in-time bound, we derive
	\begin{equation}\label{Delta u}
		\int_{t}^{t+\tau} \| \Delta u \|_{L^6}^6  ds \le c_1\quad \mathrm{for}\ \mathrm{all}\ t\in(0,\widetilde{T}_{max}).
	\end{equation}
	Then the boundedness of $ \| \nabla v\|_{L^\infty} $ can be obtained via semi-group tool. In fact, we just consider the variation-of-constants formula to the second equation of \eqref{model} over $ ((t-\tau),t) $ when $ t>\tau $. One has
	\begin{equation*}
		\begin{aligned}
			\nabla v(\cdot,t)=& \nabla e^{d_2\tau \Delta}v(\cdot,t-\tau)+\xi \int_{t-\tau}^{t} \nabla e^{d_2 (t-s)\Delta} \nabla \cdot (v\nabla u) ds\\
			&+r_2 \int_{t-\tau}^{t} \nabla e^{d_2 (t-s)\Delta} v(1-v)ds+ \int_{t-\tau}^{t} \nabla e^{d_2 (t-s)\Delta} (uv-b_2vw) ds\\
		\end{aligned}
	\end{equation*}
	for all $ t\in (\tau,T_{max}) $, which together with \eqref{SG-2} implies
	\begin{align}\label{v}
	&	\| \nabla v(\cdot,t)\|_{L^\infty}\nonumber\\
		\le&  \| \nabla e^{d_1\tau \Delta}v(\cdot,t-\tau)\|_{L^\infty}+ \xi \int_{t-\tau}^{t} \| \nabla e^{d_2 (t-s)\Delta} \nabla \cdot (v\nabla u) \|_{L^\infty}ds\nonumber\\
		&+ \int_{t-\tau}^{t} \| \nabla e^{d_2 (t-s)\Delta} [r_2v(1-v)+uv] \|_{L^\infty}ds +b_2 \int_{t-\tau}^{t}\| \nabla e^{d_2(t-s)\Delta}vw \|_{L^\infty}ds\nonumber \\
		\le& \| \nabla v_0 \|_{L^\infty}+ \xi \int_{t-\tau}^{t} \| \nabla e^{d_2(t-s)\Delta}\nabla u \cdot \nabla v \|_{L^\infty}ds+\xi \int_{t-\tau}^{t} \| \nabla e^{d_2(t-s)\Delta}v\Delta u \|_{L^\infty}ds\nonumber\\
		&+  \gamma_2 \int_{t-\tau}^{t} \left(1+(t-s)^{-\frac{1}{2}}\right)e^{-\lambda_1 d_2(t-s)}(r_2K_9(1+K_9)+K_1K_9) ds\nonumber\\
		&+ \gamma_2K_9 \int_{t-\tau}^{t} \left(1+(t-s)^{-\frac{1}{2}-\frac{1}{6}}\right)e^{-\lambda_1 d_2(t-s)} \| w \|_{L^6}ds \nonumber\\
		\le & \xi \int_{t-\tau}^{t} \| \nabla e^{d_2(t-s)\Delta}\nabla u \cdot \nabla v \|_{L^\infty}ds+\xi \int_{t-\tau}^{t} \| \nabla e^{d_2(t-s)\Delta}v\Delta u \|_{L^\infty}ds+ c_2
	\end{align}
	for all $ t\in(\tau,T_{max}) $. Now let us deal with the right-hand side of the last inequality with the help of \eqref{Delta u} and gradient estimates, it follows that
	\begin{equation*}
		\begin{aligned}
			&\xi \int_{t-\tau}^{t} \| \nabla e^{d_2(t-s)\Delta}\nabla u \cdot \nabla v \|_{L^\infty}ds\\
			&= \xi \gamma_2 \int_{t-\tau}^{t} \left(1+(t-s)^{-\frac{1}{2}-\frac{1}{3}}\right) e^{-\lambda_1 d_2 (t-s)} \| \nabla u \cdot \nabla v \|_{L^3} ds\\
			&\le \xi \gamma_2 \int_{t-\tau}^{t} \left(1+(t-s)^{-\frac{5}{6}}\right)e^{-\lambda_1 d_2 (t-s)}  \| \nabla v \|_{L^2} \| \nabla u \|_{L^6} ds\\
			& \le c_3 \left(\Gamma \left(\frac{1}{6}\right)+1\right),
		\end{aligned}
	\end{equation*}
and
	\begin{equation*}
		\begin{aligned}
			&\xi \int_{t-\tau}^{t} \| \nabla e^{d_2(t-s)\Delta}v\Delta u \|_{L^\infty}ds \\
			&=\xi \gamma_2 \int_{t-\tau}^{t} (1+(t-s)^{-\frac{1}{2}-\frac{1}{6}})e^{-\lambda_1 d_1 (t-s)}\| v \Delta u \|_{L^6}ds\\
			&\le  (\xi \gamma_2 K_1)^{\frac{6}{5}} \int_{t-\tau}^{t} (1+(t-s)^{-\frac{2}{3}})^{\frac{6}{5}}e^{-\frac{6\lambda_1 d_1 }{5}(t-s)}ds + \int_{t-\tau}^{t} \| \Delta v \|_{L^6}^6 ds\\
			&\le c_4 \left(\Gamma\left(\frac{1}{5}\right)+1\right)+c_1\quad \mathrm{for}\ \mathrm{all}\ t\in(\tau,T_{max}).
		\end{aligned}
	\end{equation*}
	
	At last, placing above two inequalities into \eqref{v} yields $ \| \nabla v\|_{L^\infty}\le c_{10} $, which in combination previous discussion ensures \eqref{K14}.
\end{proof}

\begin{lemma}\label{le-K15}
	Assume $ \sigma>1 $. Then there exists a constant $ K_{15}>0 $ independent of $ t $ such that the local classical solution $ (u,v,w) $ of \eqref{model} satisfies
	\begin{equation}\label{K15}
		\| w(\cdot,t) \|_{L^\infty} \leq K_{15}\quad \mathrm{for}\ \mathrm{all}\  t\in (0,T_{max}).
	\end{equation}
\end{lemma}
\begin{proof}
By the variation of constants formula, $ w  $ can be represented as
\begin{equation*}
	\begin{aligned}
		w(\cdot,t)=&e^{(\Delta -1)t}w_0-\chi\int_0^te^{(\Delta-1)(t-s)}\nabla \cdot ( uw\nabla v+vw\nabla u)ds\\
		&+ \int_0^te^{(\Delta-1)(t-s)}[r_3 w(1-w^\sigma)+vw+uw]ds.
	\end{aligned}
\end{equation*}
Applying the semigroup smoothing estimates \eqref{SG-1} and \eqref{SG-4}, one can
obtain
\begin{align*}
		\|w(\cdot,t)\|_{L^\infty}\le & \|e^{(\Delta 
			-1)t}w_0\|_{L^\infty}+\chi\int_0^t\|e^{(\Delta-1)(t-s)}\nabla \cdot ( 
		uw\nabla v+vw\nabla u)\|_{L^\infty}ds\nonumber\\
		&+ 
		\int_0^t\|e^{(\Delta-1)(t-s)}w(r_3+v+u)\|_{L^\infty}ds\nonumber\\
		\le& \|w_0\|_{L^\infty}+ 
		\chi\gamma_4\int_{0}^t\left(1+(t-s)^{-\frac{5}{6}}\right)
		e^{-(\lambda_1+1)(t-s)}\|w(u\nabla v+v\nabla u)\|_{L^3}ds\nonumber\\
		&+ \gamma_1\int_{0}^t\left(1+(t-s)^{-\frac{1}{3}}\right)e^{-(t-s)}\|
		w(r_3+v+u)\|_{L^3}ds\nonumber\\
		\le& \|w_0\|_{L^\infty}+\chi\gamma_4\left(\|u\|_{L^\infty}\|\nabla 
		v\|_{L^\infty}+\|v\|_{L^\infty}\|\nabla u\|_{L^\infty}\right)\|w\|_{L^3}
		\left(\Gamma\left(\frac{1}{6}\right)+1\right)\nonumber\\
		&+ 
		\gamma_1\|w\|_{L^3}\left(r_3+\|u\|_{L^\infty}+\|v\|_{L^\infty}\right)
		\left(\Gamma\left(\frac{2}{3}\right)+1\right)\nonumber\\
		\le& c_1, \quad \text{ for all } t\in(0,T_{max}),
\end{align*}
where we have used Lemma \ref{le-K1}, Lemma \ref{le-K9}, Lemma \ref{le-K14} and 
Lemma \ref{le-K15}. The proof is completed.
\end{proof}

\begin{proof}[Proof of Theorem \ref{GS}]
	Lemma \ref{le-K1} and Lemma \ref{le-K9} in combination with Lemma \ref{le-K14} show that 
	\begin{equation*}
		\|u(\cdot,t)\|_{W^{1,\infty}}+	\|v(\cdot,t)\|_{W^{1,\infty}} \le c_1\quad \mathrm{for}\ \mathrm{all}\ t\in(0,T_{max}).
	\end{equation*}
	From this and Lemma \ref{le-K15}, we can find a constant $ c_2 $ such that
	\begin{equation*}
		\| u(\cdot,t)\|_{W^{1,\infty}} +\| v(\cdot,t)\|_{W^{1,\infty}}+ 	\|w(\cdot,t)\|_{L^\infty}\le c_2
	\end{equation*}
	for all $ t\in(0,T_{max}) $, which alongside with the extension criterion in Lemma \ref{LS} proves Theorem \ref{GS}.
\end{proof}
	
	\section{Global stability of solutions} \label{4}
	This section is dedicated to the global stability of classical solutions to 
	\eqref{model}. We shall prove the three-species coexistence steady state $ 
	(u_*,v_*,w_*) $ defined by \eqref{bbb} is global stable and show that the 
	convergence rate is exponential under some conditions . To this end, we 
	introduce the following energy functional:
	\begin{equation}\label{4.1}
		\mathcal{E}(t):= \mathcal{E}(u,v,w) = \frac{1}{b_3}\mathcal{E}_u (t) +\frac{1}{b_2}\mathcal{E}_v(t)+\mathcal{E}_w(t),
	\end{equation}
where
\begin{equation*}
	\mathcal{E}_{l} (t) := \int_{\Omega} l-l_*-l_* \ln \frac{l}{l_*}, \quad l = u,v,w.
\end{equation*}

\begin{lemma}\label{lemma 4.1}
	Let the assumptions in Theorem \ref{GST} hold. If 
	\begin{align}\label{4.2222}
		(b_1b_2-b_3)^2<4b_2b_3,
	\end{align}
	then there exist $ \chi_1>0 $ and $ \xi_1>0 $ such that whenever $ 
	\chi\in(0,\chi_1) $ and $ \xi\in(0,\xi_1) $ 
	\begin{equation}
		\|u-u_*\|_{L^2}+\|v-v_*\|_{L^2}+\|w-w_*\|_{L^2}\le C_3e^{-C_4t}, \quad 
		\mathrm{for}\ \mathrm{all}\  t>t_0,
	\end{equation} 
	where $ t_0>1 $ and $ C_3,C_4>0 $ are constants independent of $ t $. 
\end{lemma}
\begin{proof}
	Using the first equation of system \eqref{model}, one can derive
	\begin{align}\label{4.2}
			\frac{d}{dt} \mathcal{E}_u(t) =&\int_{\Omega} 
			\left(1-\frac{u_*}{u}\right)u_t\nonumber\\
			=&-d_1 u_* \int_{\Omega} \frac{|\nabla u|^2}{u^2} + \int_{\Omega} 
			(u-u_*)(1-u-b_1v-b_3w)\nonumber\\
			=& -d_1 u_* \int_{\Omega}\frac{|\nabla u|^2}{u^2}+ \int_{\Omega} 
			(u-u_*)(u_*-u+1-u_*-b_1v-b_3w)\nonumber\\
			=& -d_1 u_* \int_{\Omega}\frac{|\nabla u|^2}{u^2}-\int_{\Omega} 
			(u-u_*)^2-b_1\int_{\Omega} (u-u_*)(v-v_*)\nonumber\\
			 &-b_3 \int_{\Omega} (u-u_*)(w-w_*),
	\end{align}
	where we have used the fact $ 1-u_* = b_1 v_* + b_3 w_* $ and the remaining steady state equations will also be exploited in next two calculations.
	
	Next, we employ the second and third equations of \eqref{model} to obtain that
	\begin{align}\label{4.3}
			\frac{d}{dt} \mathcal{E}_v(t) =& \int_{\Omega} 
			\left(1-\frac{v_*}{v}\right)v_t\nonumber\\
			=&-d_2 v_* \int_{\Omega} \frac{|\nabla v|^2}{v^2} + \xi v_* 
			\int_{\Omega} \frac{\nabla u \cdot \nabla v}{v} + \int_{\Omega} 
			(v-v_*)(1-v+u-b_2 w)\nonumber\\
			=& -d_2 v_* \int_{\Omega} \frac{|\nabla v|^2}{v^2} + \xi v_* 
			\int_{\Omega} \frac{\nabla u \cdot \nabla v}{v} + \int_{\Omega} 
			(v-v_*)(v_*-v+1-v_*+u-b_2 w)\nonumber\\
			=& -d_2 v_* \int_{\Omega} \frac{|\nabla v|^2}{v^2} + \xi v_* 
			\int_{\Omega} \frac{\nabla u \cdot \nabla v}{v} -\int_{\Omega} 
			(v-v_*)^2\nonumber\\
			&+ \int_{\Omega} (u-u_*)(v-v_*) -b_2 \int_{\Omega}  (v-v_*)(w-w_*)
	\end{align}
	and
	\begin{align}\label{4.4}
			\frac{d}{dt} \mathcal{E}_w(t) = & 
			\int_{\Omega}\left(1-\frac{w_*}{w}\right)w_t\nonumber\\ 
			= & -w_* \int_{\Omega} \frac{|\nabla w|^2}{w^2} + \chi w_* 
			\int_{\Omega} \frac{v \nabla u \cdot \nabla w}{w} + \chi w_* 
			\int_{\Omega} \frac{u \nabla v \cdot \nabla w}{w} \nonumber\\
			&+ \int_{\Omega} 
			(w-w_*)(w_*^\sigma-w^\sigma+1-w_*^\sigma+v+u)\nonumber\\
			= &-w_* \int_{\Omega} \frac{|\nabla w|^2}{w^2} + \chi w_* 
			\int_{\Omega} \frac{v \nabla u \cdot \nabla w}{w} + \chi w_* 
			\int_{\Omega} \frac{u \nabla v \cdot \nabla w}{w} \nonumber\\
			&-\int_{\Omega} (w-w_*)(w^\sigma-w_*^\sigma)+ \int_{\Omega} 
			(u-u_*)(w-w_*)+\int_{\Omega} (v-v_*)(w-w_*).
	\end{align}
	
	Then, substituting \eqref{4.2} \eqref{4.3} and \eqref{4.4} into \eqref{4.1}, we 
	obtain
	\begin{equation}\label{4.5}
		\begin{aligned}
			\frac{d}{dt}\mathcal{E}(t)+ \int_{\Omega} (w-w_*)(w^\sigma-w_*^\sigma)= 
			-\int_{\Omega}XAX^T-\int_{\Omega}YBY^T,
		\end{aligned}
	\end{equation}
	where $ X:= (u-u_*,v-v_*) $, $ Y:= (\frac{\nabla u}{u},\frac{\nabla 
		v}{v},\frac{\nabla w}{w}) $ and $ A,B $ are symmetric matrices defined by
	
	\begin{equation*}
		A:=
		\begin{pmatrix}
			\frac{1}{b_3} & \frac{b_1b_2-b_3}{2b_2b_3} \\
			\frac{b_1b_2-b_3}{2b_2b_3} & \frac{1}{b_2}	
		\end{pmatrix},
		\quad B:= 
		\begin{pmatrix}
			\frac{d_1u_*}{b_3} & -\frac{\xi v_*u}{2b_2} & -\frac{\chi w_*uv}{2} \\
			-\frac{\xi v_*u}{2b_2} & \frac{d_2v_*}{b_2} & -\frac{\chi w_*uv}{2} \\
			-\frac{\chi w_*uv}{2} & -\frac{\chi w_*uv}{2} & w_*
		\end{pmatrix}.
	\end{equation*}
	
	We show that $ A $ and $ B $ are positive definite. In fact,
	\begin{equation*}
		\begin{aligned}
			|A|=\frac{1}{b_2b_3}-\frac{(b_1b_2-b_3)^2}{4b_2^2b_3^2}= 
			\frac{4b_2b_3-(b_1b_2-b_3)^2}{4b_2^2b_3^2}>0
		\end{aligned}
	\end{equation*}
	by means of \eqref{4.2222}.  Therefore, $ A $ is positive definite and there 
	exists one constant $ 
	c_1>0 $ such that
	\begin{align}\label{4.6}
		XAX^T\ge c_1|X|^2.
	\end{align}
	On the other hand, one can find constants $ \xi_1>0 $ and $ \chi_1>0 $, such 
	that for all $ \xi\in(0,\xi_1) $ and $ \chi\in(0,\chi_1) $
	\begin{align*}
		\left|
		\begin{array}{cc}
			\frac{d_1u_*}{b_3} & -\frac{\xi v_*u}{2b_2} \\
			-\frac{\xi v_*u}{2b_2} & \frac{d_2v_*}{b_2}
		\end{array}
		\right|= \frac{v_*(4d_1d_2b_2u_*-\xi^2v_*u^2)}{4b_2^2}\ge 
		\frac{v_*(4d_1d_2b_2u_*-\xi^2v_*\|u\|_{L^\infty}^2)}{4b_2^2}>0,
	\end{align*}
	and 
	
	\begin{equation*}
		\begin{aligned}
			|B|=& \frac{w_*}{4b_2^2b_3}\left[4d_1d_2b_2u_*v_*-b_2\chi^2w_*(d_1b_2u_*
			+d_2b_3v_*)u^2v^2\right.\\
			&-\left.b_3\xi^2v_*^2u^2-b_2b_3\xi\chi^2w_*v_*u^3v^2\right]\\
			\ge& \frac{w_*}{4b_2^2b_3}\left[4d_1d_2b_2u_*v_*-b_2\chi^2w_*(d_1b_2u_*
			+d_2b_3v_*)\|u\|_{L^\infty}^2\|v\|_{L^\infty}^2\right.\\
			&-\left.b_3\xi^2v_*^2\|u\|_{L^\infty}^2-b_2b_3\xi\chi^2w_*v_*
			\|u\|_{L^\infty}^3\|v\|_{L^\infty}^2\right]\\
			>&0, 
		\end{aligned} 
	\end{equation*}
	since $ \|u\|_{L^\infty} $ is independent of $ \xi,\chi $ 
	and $ \|v\|_{L^\infty} $ is independent of $ \chi $. 
	Thus, $ B $ is positive definite and hence 
	\begin{align}\label{4.7}
		YBY^T>0.
	\end{align}
	
	We now substitute \eqref{4.6} and \eqref{4.7} into \eqref{4.5} and notice the 
	inequality $ (a-b)(a^q-b^q)\ge \max\{a^{q-1},b^{q-1}\}(a-b)^2 $ for $ a,b>0 $ 
	and $ q>1 $, one can find a constant $ c_2 $ such that
	\begin{equation}
		\begin{aligned}
			\frac{d}{dt}\mathcal{E}(t)+ c_2\left(\int_{\Omega} 
			(u-u_*)^2+\int_{\Omega} (v-v_*)^2+\int_{\Omega} (w-w_*)^2\right)<0.
		\end{aligned}
	\end{equation}
	Then, by the same argument as in the proof of \cite[Lemma 4.3]{JWW}, we can find 
	constants $ c_3,c_4>0 $ such that
	\begin{equation}
		\|u-u_*\|_{L^2}+\|v-v_*\|_{L^2}+\|w-w_*\|_{L^2}\le c_3e^{-c_4t}, \quad 
		\text{ for all } t>t_0.
	\end{equation}
	where $ t_0>1 $ is a constant. The proof is completed.
\end{proof}

Similar to the proof of \cite[Lemma 4.2]{JWW}, we can get the higher regularity 
of solutions as follows.
\begin{lemma}\label{lemma 4.2}
	Let $(u, v, w)$ be the unique global bounded classical solution of 
	\eqref{model} given by Theorem \ref{GS}.
	Then for any given $0<\alpha <1,$ there exists a constant $C>0$ such that
	\begin{equation*}
		\|u(\cdot,t)\|_{C^{2+\alpha, 1+\frac{\alpha}{2}}(\bar \Omega \times [1, 
			\infty))}+\|v(\cdot,t)\|_{C^{2+\alpha, 1+\frac{\alpha}{2}}(\bar \Omega 
			\times [1, \infty))}+\|w(\cdot,t)\|_{C^{2+\alpha, 
				1+\frac{\alpha}{2}}(\bar \Omega \times [1, \infty))}\leq C.
	\end{equation*}
\end{lemma}

With the above lemmas in hand, we now prove Theorem \ref{GST}.
\begin{proof}[Proof of Theorem \ref{GST}]
	By Lemma \ref{lemma 4.1}, we obtain 
	\begin{equation*}
		\|u-u_*\|_{L^2}+\|v-v_*\|_{L^2}+\|w-w_*\|_{L^2}\le c_1e^{-c_2t}, 
		\quad 
		\text{ for all } t>t_0.
	\end{equation*} 
	Using Lemma \ref{lemma 4.2}, one has
	\begin{align*}
		\|u-u_*\|_{W^{1,\infty}}+\|v-v_*\|_{W^{1,\infty}}+\|w-w_*\|
		_{W^{1,\infty}}\le c_3, 
		\quad 
		\text{ for all } t>1.
	\end{align*}
	Then applying Gagliardo-Nirenberg inequality, we get
	\begin{align*}
		\|u-u_*\|_{L^\infty}\le 
		c_4\|u-u_*\|_{W^{1,\infty}}^\frac{1}{2}\|u-u_*\|_{L^2}^\frac{1}{2}\le 
		c_1^{\frac{1}{2}}c_3^{\frac{1}{2}}c_4e^{-\frac{c_2}{2}t}, \quad \text{ for 
			all } t>t_0.
	\end{align*}
	Similarly, we obtain
	\begin{align*}
		\|u-u_*\|_{L^\infty}+\|v-v_*\|_{L^\infty}+\|w-w_*\|_{L^\infty}\le 
		c_5e^{-c_6t}, \quad \text{ for 
			all } t>t_0.
	\end{align*}
	This completes the proof of Theorem \ref{GST}.
\end{proof}

\end{document}